\documentclass[11pt]{amsart}
\usepackage{amsthm,amsmath,amssymb,enumerate}
\title[]{On the modified Futaki invariant of complete intersections in projective spaces}

\author[]{Ryosuke Takahashi}

\address{Mathematical Institute, Tohoku University, 6-3, Aoba, Aramaki, Aoba-ku, Sendai, 980-8578, Japan}

\email{ryosuke.takahashi.a7@tohoku.ac.jp}

\keywords{Fano variety, K\"ahler-Ricci soliton, complete intersection}

\thanks{This is the author's accepted version. The final publication is available at Cambridge University Press via http://dx.doi.org/10.1017/nmj.2016.16.}

\subjclass[2010]{53C25}
\makeatletter

\@addtoreset{equation}{section}
\makeatother

{
\theoremstyle{definition}
\newtheorem{defi}{Definition}[section]
\newtheorem{theo}[defi]{Theorem}
\newtheorem{prop}[defi]{Proposition}
\newtheorem{lemm}[defi]{Lemma}
\newtheorem{coro}[defi]{Corollary}
\newtheorem{rema}[defi]{Remark}
\newtheorem{exam}[defi]{Example}
}

{
\theoremstyle{definition}
\newtheorem*{ackn}{Acknowledgements}
}

\begin{document}
\begin{abstract}
Let $M$ be a Fano manifold. We call a K\"ahler metric $\omega \in c_1(M)$ a K\"ahler-Ricci soliton if it satisfies the equation ${\rm Ric}(\omega) - \omega = L_V \omega$ for some holomorphic vector field $V$ on $M$. It is known that a necessary condition for the existence of K\"ahler-Ricci solitons is the vanishing of the modified Futaki invariant introduced by Tian-Zhu. In a recent work of Berman-Nystr\"om, it was generalized for (possibly singular) Fano varieties and the notion of algebro-geometric stability of the pair $(M,V)$ was introduced. In this paper, we propose a method of computing the modified Futaki invariant for Fano complete intersections in projective spaces.
\end{abstract}
\maketitle
\tableofcontents
\section{Introduction}
Let $M$ be an $n$-dimensional Fano manifold, i.e., $M$ is a compact complex manifold and $c_1(M)$ is represented by some K\"ahler form $\omega$ on $M$.
If we take holomorphic coordinates $(z^1, \ldots , z^n)$ of $M$, $\omega$ and its Ricci form ${\rm Ric}(\omega)$ are locally written as
\[
\begin{cases}
g_{i \bar{j}} = g \left( \frac{\partial}{\partial z^i }, \frac{\partial}{\partial z^{\bar j} } \right) \\
\omega = \frac{\sqrt{-1}}{2 \pi} \sum_{i, j} g_{i \bar{j}} dz^i \wedge dz^{\bar j}
\end{cases}
\]
and
\[
\begin{cases}
r_{i \bar{j}} = - \partial_i \partial_{\bar j} \log ( \det ( g_{k \bar{l}}) ) \\
{\rm Ric} (\omega) = \frac{\sqrt{-1}}{2 \pi} \sum_{i, j} r_{i \bar{j}} dz^i \wedge dz^{\bar j} .
\end{cases}
\]
Since both $\omega$ and ${\rm Ric}(\omega)$ are in $c_1(M)$, ${\rm Ric}(\omega) - \omega$ is an exact $(1,1)$-form. So there exists a real-valued smooth function $\kappa$ on $M$ such that
\[
{\rm Ric}(\omega) - \omega =  \frac{\sqrt{-1}}{2 \pi} \partial \bar{\partial} \kappa .
\]
Let ${\mathfrak g}$ be the Lie algebra consisting of all holomorphic vector fields on $M$. Then any $V \in {\mathfrak g}$ can be lifted to the anti-canonical bundle $-K_M$ of $M$, and naturally acts on the space of Hermitian metrics on $-K_M$. Let $h$ be a Hermitian metric on $-K_M$ such that $\omega = - \frac{\sqrt{-1}}{2 \pi} \partial \bar{\partial} \log h$ and $\mu_{h, V}$ the holomorphy potential of the pair $(h, V)$ defined by this action (cf. Definition \ref{defi:2.2}). Then we can easily check that
\[
\begin{cases}
i_V \omega = \frac{\sqrt{-1}}{2 \pi} \bar{\partial} \mu_{h,V} \\
- \Delta_{\partial} \mu_{h,V} +  \mu_{h,V} + V(\kappa) = 0 ,
\end{cases}
\]
where $\Delta_{\partial} = - g^{i \bar{j}} \frac{\partial^2}{\partial z^i \partial z^{\bar j}}$ denotes the $\partial$-Laplacian with respect to $\omega$.
A metric $\omega$ is called a K\"ahler-Ricci soliton if it satisfies the equation
\[
{\rm Ric}(\omega) - \omega = L_V \omega
\]
for some $V \in {\mathfrak g}$, where $L_V$ denotes the Lie derivative with respect to $V$. This is equivalent to the condition $\kappa = \mu_{h,V}$ (up to an additive constant). Especially, in the case when $V \equiv 0$, this metric is a well-known K\"ahler-Einstein metric. An obstruction to the existence of K\"ahler-Ricci solitons was first discovered by Tian-Zhu \cite{TZ02}: let ${\mathcal F}$ be a function on ${\mathfrak g}$ defined by
\[
{\mathcal F}(V) = - \frac{1}{c_1(M)^n} \int_M e^{\mu_{h,V}} \omega^n ,
\]
and define the modified Futaki invariant ${\rm Fut}_V (W)$ as the G{\^a}teaux differential of ${\mathcal F}$ at $V$ in the direction $W$, i.e.,
\begin{eqnarray*}
{\rm Fut}_V (W) &=& \left. \frac{d}{dt} {\mathcal F}(V+t W) \right|_{t=0}  = - \frac{1}{c_1 (M)^n} \int_M  \mu_{h,W} e^{\mu_{h,V}} \omega^n \\
&=& \frac{1}{c_1 (M)^n} \int_M W(\kappa - \mu_{h,V}) e^{\mu_{h,V}} \omega^n.
\end{eqnarray*}
Hence if there exists a K\"ahler-Ricci soliton $\omega$ with respect to $V$, then we have $\kappa =  \mu_{h,V}$ (up to an additive constant) and ${\rm Fut}_V (W)$ must vanish.
They showed that ${\rm Fut}_V (W)$ is independent of a choice of $\omega \in c_1(M)$ (In the case when $V \equiv 0$, this function coincides with the original Futaki invariant and its independence was shown in \cite{Fut83}). Recently, Berman-Nystr\"om \cite{BN14} generalized this obstruction to arbitrary Fano varieties (i.e., projective normal varieties with log terminal singularities and satisfying the property that $-K_M$ is an ample ${\mathbb Q}$-line bundle) and introduced the notion of K-stability for the pair $(M,V)$ (Wang-Zhou-Zhu \cite{WZZ14} also defined the slightly modified notion of K-stability inspired by the algebraic formula for the modified Futaki invariant in \cite{BN14}). Examing the sign of the modified Futaki invariant is important, since we can know whether $c_1 (M)$ contains a K\"ahler-Ricci soliton or not if we examine the sign of the modified Futaki invariant on the central fiber for any special test configuration, i.e., check the K-polystability.

Chen-Donaldson-Sun \cite{CDS15} and Tian \cite{Tian15} proved that if $M$ is K-polystable, there exists a K\"ahler-Einstein metric.
In the case of K\"ahler-Ricci solitons, Berman-Nystr\"om \cite{BN14} showed that if $M$ admits a K\"ahler-Ricci soliton with respect to $V$, then $(M,V)$ is K-polystable. They also showed that if $M$ is strongly analytically K-polystable and all the higher order modified Futaki invariants of $(X,V)$ vanish, then there exists a  K\"ahler-Ricci soliton with respect to $V$, where ``strongly analytically K-polystable'' means that the modified K-energy is coercive modulo automorphisms. However, it is still an open question whether the K-polystability of $(M,V)$ leads to the existence of a K\"ahler-Ricci soliton with respect to $V$. 

Motivated by the above reasons, we propose a method of calculating the function ${\mathcal F}$ (therefore, the modified Futaki invariant ${\rm Fut}_V$ as well) for Fano complete intersections in projective spaces. The main theorem of this paper is:
\begin{theo}
\label{theo:1.1}
Let $M$ be a Fano complete intersection in ${\mathbb CP}^N$, i.e., $M$ is an $(N-s)$-dimensional Fano variety in ${\mathbb CP}^N$ defined by homogeneous polynomials $F_1, \ldots , F_s$ of degree $d_1, \ldots , d_s$ respectively, and $\omega = \frac{\sqrt{-1}}{2 \pi} \partial \bar{\partial} \log \left( \sum_{i=0}^N |z^i|^2 \right)$ the Fubini-Study metric of ${\mathbb CP}^N$. We suppose that there exists a constant $m > 0$ such that $m \omega \in c_1 (M)$. Let $V \in {\mathfrak sl}(N+1, {\mathbb C})$ be a holomorphic vector field on ${\mathbb CP}^N$ such that $V F_i =\alpha_i F_i$ for some constants $\alpha_i$ $(i = 1, \ldots , s)$. Then we have $m= N+1-d_1 - \cdots - d_s$ and the function ${\mathcal F}$ can be written as
\begin{equation}
{\mathcal F}(V)=- \frac{(N-s)!}{d_1 \cdots d_s m^{N-s}} \exp \left( \sum_{i=1}^s \alpha_i \right) \int_{{\mathbb CP}^N} \prod_{i=1}^s (d_i \omega + d_i \theta_V - \alpha_i) e^{m \theta_V} \cdot e^{m \omega} \label{eq:1.1},
\end{equation}
where $\theta_V := V \log \left( \sum_{i=0}^N |z^i|^2 \right)$.
\end{theo} 
From the above theorem, we know that ${\mathcal F}(V)$ can be written as a linear combination of the integrals $I_{0,l}:= m^l \int_{{\mathbb CP}^N} (\theta_V)^l e^{m \theta_V} \omega^N$ ($0 \leq l \leq s$).

Though we can easily get a method of computing ${\mathcal F}$ using the localization formula for orbifolds in \cite{DT92}, our formula \eqref{eq:1.1} is still valuable since we need not to assume that $M$ has at worst orbifold singularities. And we also do not require the explicit geometric knowledge of $M$, $V$ and $\omega$ (local coordinates (uniformization), the zero set of $V$, curvature, etc.).
More concretely, in order to apply the localization formula in \cite{DT92} directly to our case, we have to know:\\
(1) The zero set ${\rm Zero}(V)$ of $V$, where we assume that ${\rm Zero}(V)$ consists of disjoint nondegenerate submanifolds $\{ Z_i \}$.\\
(2) The values of integrals
\[
\int_{Z_i} \frac{ e^{m(\omega + \theta_V)}}{\det (L_{i,V} + K_i)} ,
\]
where $L_{i,V}(W) := [V, W]$ denotes an endomorphism and $K_i$ the curvature matrix of the normal bundle of $Z_i$.

If $s ( = {\rm codim}(M)) = 1$ and ${\rm dim}(Z_i) = 0$, the above integral can be computed by taking local coordinates (or uniformization) around $Z_i$. However, it is very hard to compute in general.

The Futaki invariant of complete intersection was first computed by Lu \cite{Lu99} using the adjunction formula and the Poincar\'e-Lelong formula. Then it was also computed by many mathematicians using different techniques (\cite{PS04}, \cite{Hou08} and \cite{AV11}). Lu \cite{Lu03} also computed the modified Futaki invariant for smooth hypersurfaces in projective spaces. Our formula (Theorem \ref{theo:1.1}) extends the Lu's result \cite{Lu03} for (possibly singular) Fano complete intersections of arbitrary codimension. Compared to the K\"ahler-Einstein case \cite{Lu99}, our formula has in common in that ${\mathcal F}(V)$ is expressed by the degree $d_1, \ldots, d_s$ of defining polynomials of $M$ and the weights $\alpha_1, \ldots, \alpha_s$ of the actions induced by the vector field $V$. However, we need more knowledge of $V$ to compute the integrals $I_{0,l}$ ($0 \leq l \leq s$) (see \S.5 for more details).

In this paper, we prove the main theorem (Theorem \ref{theo:1.1}) based on the calculations in \cite{Lu99} and \cite{AV11}. 
In \S.2, we review some fundamental materials and results for K\"ahler-Ricci solitons. The standard reference for (holomorphic) equivariant cohomology theory are \cite{BGV92}, \cite{Hou08} and \cite{Liu95}. We introduce an algebraic formula for ${\mathcal F}$ in reference to the quantization of the modified Futaki invariant studied in \cite{BN14}. In \S.3, we give a proof of Theorem \ref{theo:1.1} by the Poincar\'e-Lelong formula. Then, in \S.4, we also give another proof of Theorem \ref{theo:1.1} using the algebraic formula for ${\mathcal F}$ (cf. Proposition \ref{prop:2.8}). Finally, we give examples of computation of ${\mathcal F}$ in \S.5.
\begin{ackn}
The author would like to express his gratitude to Professor Ryoichi Kobayashi for his advice on this article, and to the referee for useful suggestions that helped him to improve the original manuscript. The author is supported by Grant-in-Aid for JSPS Fellows Number 25-3077.
\end{ackn}
\section{Preminaries}
\subsection{Holomorphic equivariant coholomogy}
Let $M$ be a complex manifold and $G$ be a Lie group acting holomorphically on $M$. Denote ${\mathfrak g}:={\rm Lie}(G)$ the Lie algebra of $G$. Then, for each $\xi \in {\mathfrak g}$, we denote by $\xi_M^{\mathbb R}$, the real holomorphic vector field on $M$ given by
\[
\xi_M^{\mathbb R} (f)(p) = \left. \frac{d}{dt} f(\exp(-t \xi) \cdot p) \right|_{t=0} \; , f \in C^{\infty}(M), \; p \in M.
\]
and $\xi_M := \frac{1}{2} (\xi_M^{\mathbb R} - \sqrt{-1} J \xi_M^{\mathbb R})$, the complex holomorphic vector field on $M$.
Let ${\mathbb C}[{\mathfrak g}]$ be the algebra of complex valued polynomial function on ${\mathfrak g}$. We regard each element in ${\mathbb C}[{\mathfrak g}] \otimes {\mathcal A}(M)$ as a polynomial function which takes values in differential forms. The group $G$ acts on an element $\sigma \in {\mathbb C}[{\mathfrak g}] \otimes {\mathcal A}(M)$ by
\[
(g \cdot \sigma)(\xi) = g \cdot ( \sigma (g^{-1} \cdot \xi)) \; , \text{$g \in G$ and $\xi \in {\mathfrak g}$}.
\]
Let ${\mathcal A}_G(M)= ({\mathbb C}[{\mathfrak g}] \otimes {\mathcal A}(M))^G$ be the space of $G$-invariant elements in ${\mathbb C}[{\mathfrak g}] \otimes {\mathcal A}(M)$. For $ \sigma \in {\mathbb C}[{\mathfrak g}] \otimes {\mathcal A}(M)$, we define the bidegree of $\sigma$ by
\[
{\rm bideg}(\sigma) = ({\rm deg (P)} + p, {\rm deg}(P)+q),
\]
where $\sigma = P \otimes \varphi$ ($P \in {\mathbb C}[{\mathfrak g}]$ and $\varphi \in {\mathcal A}^{p,q}(M)$). For instance, ${\rm bideg}(\xi) = (1,1)$. Thus, ${\mathcal A}_G (M) = \bigoplus{\mathcal A}_G^{p,q}(M)$ has a structure of a bigraded algebra. We define the equivariant exterior differential $\bar{\partial}_{\mathfrak g}$ on ${\mathbb C}[{\mathfrak g}] \otimes {\mathcal A}(M)$ as
\[
(\bar{\partial}_{\mathfrak g} \sigma)(\xi) = \bar{\partial} (\sigma(\xi))+2 \pi \sqrt{-1} i_{\xi_M} (\sigma (\xi)) , \; \sigma \in {\mathbb C}[{\mathfrak g}] \otimes {\mathcal A}(M).
\]
Then $\bar{\partial}_{\mathfrak g}$ increases by $(0,1)$ the total bidegree on ${\mathbb C}[{\mathfrak g}] \otimes {\mathcal A}(M)$, and preserves ${\mathcal A}_G(M)$. Hence we have a complex $({\mathcal A}_G(M), \bar{\partial}_{\mathfrak g})$.
\begin{defi}
\label{defi:2.1}
The holomorphic equivariant cohomology $H_{\mathfrak g} (M)$ of the pair $(M, G)$ is the cohomology of the complex $({\mathcal A}_G(M), \bar{\partial}_{\mathfrak g})$.
\end{defi}
Let $E$ be a $G$-linearized holomorphic vector bundle over $M$, and ${\rm Herm}(E)$ the space of Hermitian metrics on $E$. The group $G$ acts on ${\rm Herm}(E)$ by the formula
\[
(g \cdot h)(u,v) = h(g^{-1} \cdot u, g^{-1} \cdot v ), \; \text{$g \in G$ and $u, v \in E$}.
\]
Hence for $\xi \in {\mathfrak g}$, we define the real Lie derivative of ${\mathfrak g}$ on ${\rm Herm}(E)$ by
\[
L_{\xi}^{\mathbb R} h = \left. \frac{d}{dt} \exp(t \xi) \cdot h \right|_{t=0}
\]
and the complex Lie derivative of ${\mathfrak g}$ on ${\rm Herm}(M)$ by
\[
L_{\xi} h = \frac{1}{2}(L_{\xi}^{\mathbb R} h - \sqrt{-1} L_{J \xi}^{\mathbb R} h).
\]
We can also define the representation of ${\mathfrak g}$ on the space of sections $\Gamma (E)$ in a similar way. Let $\nabla$ be the Chern connection with respect to $h$, and put
\[
\mu_{h, \xi} = L_{\xi} - \nabla_{\xi_M}.
\]
Since $\mu_{h, \xi} (fs) = \xi_M f \cdot s + f \cdot L_{\xi} s - \xi_M f \cdot s - f \cdot  \nabla_{{\xi}_M} s = f \cdot \mu_{h, \xi} (s) $ for any $f \in C^{\infty}(M)$ and $ s \in \Gamma (E)$, we have $\mu_{h,\xi} \in \Gamma({\rm End}(E))$. Moreover, one can show that
\[
L_{\xi} h = - \mu_{h, \xi} \cdot h,\;  i_{\xi_M} \theta (h) = - \mu_{h, \xi}, \; \text{and} \; \; i_{\xi_M} \Theta(h) = \frac{\sqrt{-1}}{2 \pi}\bar{\partial} \mu_{h, \xi},
\]
where $\theta (h) = \partial h \cdot h^{-1}$ is the connection form and $\Theta(h) = \frac{\sqrt{-1}}{2 \pi} \bar{\partial} (\partial h \cdot h)$ is the curvature form with respect to $h$. Define the equivariant curvature form $\Theta_{\mathfrak g} (h)$ by
\[
\Theta_{\mathfrak g} (h) = \Theta (h) + \mu_{h, \xi},
\]
Then $\Theta_{\mathfrak g} (h)$ is $\bar{\partial}_{\mathfrak g}$-closed and defines an element in $H_{\mathfrak g}^{1,1} (M)$.

Now, let us consider the case when $E=L$ is a $G$-linearized ample line bundle. Then $\mu_{h, \xi}$ is a complex valued smooth function on $M$.
\begin{defi}
\label{defi:2.2}
The function $\mu_{h, \xi}$ is said to be the holomorphy potential of the pair $(h, \xi)$.
\end{defi}
\subsection{K\"ahler-Ricci soliton}
Let $M$ be an $n$-dimensional Fano manifold.
\begin{defi}
\label{defi:2.3}
A K\"ahler metric $\omega$ on $M$ is a K\"ahler-Ricci soliton if the metric $\omega$ solves the equation
\begin{equation}
{\rm Ric}(\omega) -\omega = L_V \omega \label{eq:2.1}
\end{equation}
for some holomorphic vector field $V$ on $M$.
\end{defi}
If the pair $(\omega,V)$ is a K\"ahler-Ricci soliton, taking the imaginary part of \eqref{eq:2.1} yields $L_{{\rm Im}(V)} \omega = 0$, so, $\omega$ is invariant under the group action generated by ${\rm Im}(V)$. More generally, we have
\begin{prop}[Lemma 2.13 in \cite{BN14}]
\label{prop:2.4}
Let $M$ be a Fano manifold and $V$ a holomorphic vector field on $M$. If there exists a K\"ahler metric $\omega$ which is invariant under the action of ${\rm Im}(V)$, then there exists a complex torus $T_c$ acting holomorphically on $M$ such that ${\rm Im}(V)$ may be identified with an element in the Lie algebra of the corresponding real torus $T \subset T_c$.
\end{prop}
\begin{proof}
First, we check that the isometry group  $K$ of $\omega$ is a compact Lie group. This is shown by considering the canonical imbedding $M \hookrightarrow H^0(M, -kK_M)$ and the $K$-invariant Hilbert norm $||s||^2 := \int_M |s|_k ^2 \omega^n$ ($s \in H^0(M, -kK_M)$). Actually, $K$ is identified with a subgroup of the group consisting of unitary transformations on $H^0(M, -kK_M)$ with respect to $|| \cdot ||$, which yields $K$ is compact. 
Taking the topological closure of the 1-parameter subgroup generated by ${\rm Im}(V)$ in $K$, we get a real torus $T$ as desired.
In general, any holomorphic action of a real torus on $M$ can be naturally extended to the corresponding complex torus action on $M$. 
\end{proof}
\subsection{Modified Futaki invariant}
Let $M$ be an $n$-dimensional Fano variety. For simplicity, let us make the following assumptions:\\
(1) $M$ is a compact subvariety of a projective manifold $N$.\\
(2) $L$ is an ample line bundle on $N$ such that on the regular part $M_{\rm reg}$ of $M$, the isomorphism
\begin{equation}
L |_{M_{\rm reg}} \simeq - k K_{M_{\rm reg}} \label{eq:2.2}
\end{equation}
holds for some integer $k$.\\
(3) The Lie group $G:={\rm Aut}(M)$ acts on $(N,L)$ such that the isomorphism \eqref{eq:2.2} is $G$-equivariant.
\begin{rema}
In fact, $M$ can be embedded into ${\mathbb CP}^N \simeq {\mathbb P} H^0(M, -k K_M)^{\ast}$ for a sufficient large $k$, and $( {\mathbb CP}^N, {\mathcal O}(1))$ satisfies the requirement above.\\
\end{rema}
We say that $V$ is a holomorphic vector field on a Fano variety $M$ if $V$ is a holomorphic vector field defined only on its regular part $M_{\rm reg}$. Then $V$ induces a local one parameter family of automorphisms, which extends to a family of $G$ since ${\rm codim}(M \backslash M_{\rm reg}) \geq 2$ by the normality of $M$ (cf: \cite[Lemma 5.2]{BBEGZ12}). Thus by the assumption (3), $V$ is given as the restriction of some holomorphic vector field on $N$ to $M$.\footnote{Such a vector field was called an ``admissible vector field'' in \cite[Definition 1.2]{DT92}. But the above argument implies that every holomorphic vector field on $M_{\rm reg}$ is automatically admissible (see also \cite[Remark 5.3]{BBEGZ12}).}
\begin{defi}
\label{defi:2.6}
A Hermitian metric $h$ on $-K_{M_{\rm reg}}$ is said to be admissible if $h^k$ can be extended to a Hermitian metric $\tilde{h}_L$ on $L$ over $N$ under the isomorphisms \eqref{eq:2.2}.
\end{defi}
Let $h$ be an admissible Hermitian metric on $-K_{M_{\rm reg}}$ and put $\omega := -\frac{\sqrt{-1}}{2 \pi} \partial \bar{\partial} \log h$. 
For holomorphic vector fields $V, W$, we define the function ${\mathcal F}$ as
\begin{equation}
{\mathcal F}(V) = - \frac{1}{c_1(M)^n}\int_{M_{\rm reg}} e^{\mu_{h, V}} \omega^n \label{eq:2.3}
\end{equation}
and the modified Futaki invariant ${\rm Fut}_V$ by
\begin{equation}
{\rm Fut}_V (W) = \left. \frac{d}{dt} {\mathcal F}(V+t W) \right|_{t=0} = - \frac{1}{c_1(M)^n} \int_{M_{\rm reg}} \mu_{h, W} e^{\mu_{h, V}} \omega^n \label{eq:2.4},
\end{equation}
where $\mu_{h, V}$ denotes the holomorphy potential of $(h, V)$ defined on $M_{\rm reg}$. Since the construction of equivariant Chern curvature form is local, if $i \colon M_{\rm reg} \hookrightarrow N$ is the embedding, we obtain
\begin{eqnarray*}
{\mathcal F}(V) &=& - \frac{1}{c_1(M)^n} \int_{M_{\rm reg}} P(\Theta_{\mathfrak g} (h, -K_{M_{\rm reg}})) \\
&=& - \frac{1}{c_1(M)^n} \int_{M_{\rm reg}} P \left( i^{\ast} \frac{\Theta_{\mathfrak g} (\tilde{h}_L, L)}{k} \right) \\
&=& - \frac{1}{c_1(M)^n} \int_{M_{\rm reg}} P \left(\frac{\Theta_{\mathfrak g} (\tilde{h}_L, L)}{k} \right),
\end{eqnarray*}
where $P(z):= n! e^z$, and this shows that the integral \eqref{eq:2.3} is finite. Moreover, using the equivariant Chern-Weil theorem, we can show the following:
\begin{theo}[\cite{Hou08}, Section 2.3]
\label{theo:2.7}
The functions ${\mathcal F}$ and ${\rm Fut}_V$ are independent of the embedding $M \hookrightarrow N$ and the choice of an admissible Hermitian metric $h$ on $-K_{M_{\rm reg}}$.
\end{theo}
On the other hand, a pluripotential theoretical formulation of ${\rm Fut}_V$ was introduced by Berman-Nystr\"om \cite{BN14}. They also introduced the quantized version of the modified Futaki invariant, which is defined more algebraically in terms of the commuting action on the cohomology $H^0(M, -kK_M)$: let $V$ be a holomorphic vector field on $M$ generating a torus action and put
\[
N_k := {\rm dim} (H^0(M, -kK_M)). 
\]
We define the quantization of the function ${\mathcal F}$ at level $k$ as
\begin{equation}
{\mathcal F}_k (V) := -k {\rm Trace}(e^{V/k})_{H^0 (M, -k K_M)} = -k \sum_{i=1}^{N_k} \exp (v_i^{(k)}/k),
\end{equation}
where $(v_i^{(k)})$ are the joint eigenvalues for the action of ${\rm Re}(V)$ on $H^0 (M, -k K_M)$ defined by the canonical lift of $V$ to $-K_M$. Additionally, let $W$ be a holomorphic vector field on $M$ generating a ${\mathbb C}^*$-action and commuting with $V$. We define the quantization of ${\rm Fut}_V (W)$ at level $k$ as
\begin{equation}
{\rm Fut}_{V,k} (W) := \left. \frac{d}{dt} {\mathcal F}_k (V+tW) \right|_{t=0} = - \sum_{i=1}^{N_k} \exp (v_i^{(k)}/k) w_i^{(k)} ,
\end{equation}
where $(v_i^{(k)}, w_i^{(k)})$ are the joint eigenvalues for the commuting action of ${\rm Re}(V)$ and ${\rm Re}(W)$.
Then we have
\begin{prop}
\label{prop:2.8}
In the case when $M$ is smooth,\\
(1) We have the asymptotic expansion of ${\mathcal F}_k (V)$ as $k \rightarrow \infty$:
\[
{\mathcal F}_k (V) = {\mathcal F}^{(0)} (V) \cdot k^{n+1} + {\mathcal F}^{(1)} (V) \cdot k^n + \cdots,
\]
where ${\mathcal F}^{(0)} (V)$ is proportional to ${\mathcal F}(V)$.\\
(2) We have the asymptotic expansion of ${\rm Fut}_{V,k} (W)$ as $k \rightarrow \infty$:
\[
{\rm Fut}_{V,k} (W) = {\rm Fut}_V^{(0)} (W) \cdot k^{n+1} + {\rm Fut}_V^{(1)} (W) \cdot k^n + \cdots,
\]
where ${\rm Fut}_V^{(i)}(W)$ is the $i$ th order modified Futaki invariant defined in \cite[\S.4.4]{BN14}, and ${\rm Fut}_V^{(0)}(W)$ is proportional to ${\rm Fut}_V (W)$.\\
(3) the $i$ th order modified Futaki invariant ${\rm Fut}_V^{(i)}(W)$ is the G{\^a}teaux differential of ${\mathcal F}^{(i)}$ at $V$ in the direction $W$, i.e., 
\[
\left. \frac{d}{dt} {\mathcal F}_k^{(i)} (V+t W) \right|_{t=0} = {\rm Fut}_V^{(i)}(W) .
\]

In general, when $M$ is a (possibly singular) Fano variety, we have\\
(4)
\[
{\mathcal F}(V) = \lim_{k \rightarrow \infty} \frac{1}{k N_k} {\mathcal F}_k (V).
\]
(5)
\[
{\rm Fut}_V (W) = \lim_{k \rightarrow \infty} \frac{1}{k N_k} {\rm Fut}_{V,k}(W).
\]
\end{prop}
\begin{proof}
The statements (2) and (5) were shown in \cite[\S.4.4]{BN14}. (3) is trivial from the definition of ${\rm Fut}_{k,V}(W)$.\\
(1) As with the proof of (2) (cf: \cite[\S.4.4]{BN14}) or \cite[Lemma 1.2]{WZZ14}, ${\mathcal F}_k (V)$ can be calculated by the equivariant Riemann-Roch formula as
\begin{eqnarray*}
{\mathcal F}_k (V) &=& -k {\rm Trace}(e^{V/k})_{H^0 (M, -k K_M)} \\
&=& -k \int_M {\rm ch}^{\mathfrak g}(-k K_M) {\rm td}^{\mathfrak g}(M) \\
&=& -k \int_M e^{\mu_{h,V}} \cdot e^{k \omega} {\rm td}^{\mathfrak g} (M) \\
&=& - \frac{1}{n!} \int_M e^{\mu_{h,V}} \omega^n \cdot k^{n+1} + O(k^n),
\end{eqnarray*}
where ${\rm ch}^{\mathfrak g}$ (resp. ${\rm td}^{\mathfrak g}$) denotes the equivariant Chern character (resp. the equivariant Todd class).
Thus, ${\mathcal F}^{(0)} (V) = \frac{c_1 (M)^n}{n!} \cdot {\mathcal F}(V)$.\\
(4) By definition, ${\mathcal F}(V)$ can be written as
\[
{\mathcal F}(V) =  - \frac{1}{c_1(M)^n}\int_M e^{\mu_{h, V}} \omega^n = - \int_{\mathbb R} e^v \nu^V ,
\]
where $\nu^V$ is the push forward measure of the Monge-Amp{\`e}re measure $\frac{\omega^n}{c_1(M)^n}$ under $\mu_{h,V}$.
Let $\nu_k ^V$ be the spectral measure on ${\mathbb R}$ attached to the infinitesimal action of ${\rm Re}(V)$ on $H^0(M,-k K_M)$:
\[
\nu_k ^V =\frac{1}{N_k} \sum_{i=1}^{N_k} \delta_{v_i ^{(k)}/k},
\]
where $\delta_{v_i ^{(k)}/k}$ denotes the Dirac measure at $v_i ^{(k)}/k$.
Then, by \cite[Proposition 4.1]{BN14}, $\nu_k ^V$ converges to $\nu^V$ as $k \rightarrow \infty$ in a weak topology. Hence we have
\[
\frac{1}{k N_k} {\mathcal F}_k (V) = -\frac{1}{N_k} \sum_{i=1}^{N_k} \exp (v_i^{(k)}/k) = - \int_{\mathbb R} e^v \nu_k ^V \rightarrow - \int_{\mathbb R} e^v \nu ^V = {\mathcal F} (V)
\]
as $k \rightarrow \infty$.
\end{proof}
\begin{rema}
\label{rema:2.9}
When $M$ is smooth, by the equivariant Riemann-Roch formula, we have an asymptotic expansion as $k \rightarrow \infty$:
\begin{equation}
N_k = \frac{1}{n!} c_1 (M)^n \cdot k^n + O(k^{n-1}) \label{eq:2.7}.
\end{equation}
Combining with Proposition \ref{prop:2.8} (1), we have
\begin{equation}
\frac{1}{k N_k} {\mathcal F}_k (V) = {\mathcal F}(V) + O(k^{-1}) \label{eq:2.8}
\end{equation}
as $k \rightarrow \infty$. In general, when $M$ is a (possibly singular) Fano variety, we do not know whether we can obtain the expansion \eqref{eq:2.8}. However, Proposition \ref{prop:2.8} (4) allows us to use the equivariant Riemann-Roch formula formally to compute the leading term of \eqref{eq:2.8} (i.e., the limit $\lim_{k \rightarrow \infty} \frac{1}{k N_k} {\mathcal F}_k (V)$) even if $M$ has singularities.
\end{rema}
\section{The calculation of the function ${\mathcal F}$}
Let $M$ be an $n$-dimensional variety in ${\mathbb CP}^N$ and $X$ a holomorphic vector field on ${\mathbb CP}^N$. Then $X$ can be identified with a linear vector field $\sum_{i,j=0}^N a_{ij} z^i \frac{\partial}{\partial z^j}$ on ${\mathbb C}^{N+1}$, and the traceless matrix $(a_{ij})_{0 \leq i,j \leq N} \in {\mathfrak sl}(N+1,{\mathbb C})$ such that the push-foward of $\sum_{i,j=0}^N a_{ij} z^i \frac{\partial}{\partial z^j}$ with the standard projection $\pi : {\mathbb C}^{N+1} - \{ 0 \} \rightarrow {\mathbb CP}^N$ is equal to $X$. 

For a holomorphic vector field $X$, we define a complex valued smooth function on ${\mathbb C}^{N+1}-0$ by
\begin{equation}
\theta_X := X \left( \log \left( \sum_{i=0}^N |z^i|^2 \right) \right) ,
\end{equation}
which descends to a smooth function on ${\mathbb CP}^N$. Let $\omega = \frac{\sqrt{-1}}{2 \pi} \partial \bar{\partial} \log (\sum_{i=1}^N |z^i|^2) \in c_1({\mathcal O}(1))$ be the Fubini-Study metric of ${\mathbb CP}^N$. Then we have
\begin{equation}
i_X \omega = \frac{\sqrt{-1}}{2 \pi} \bar{\partial} \theta_X . \label{eq:3.2}
\end{equation}

We say that ``$X$ is tangent to $M$'' if ${\rm Re}(X)$ leaves $M$ invariant. If $M$ is a hypersurface defined by a homogenous polynomial $F$ of degree $d$, $X$ is tangent to $M$ iff $X$ fixes $[F] \in {\mathbb P}(H^0(M, {\mathcal O}(d)))$, or, equivalently, $XF=\gamma F$ for some constant $\gamma$. For any $X$ which is tangent to $M$, the equation \eqref{eq:3.2} can be written as
\begin{equation}
X^i = g^{i \bar{j}} \frac{\partial \theta_X}{\partial x^{\bar j}} \; (i=1, \ldots , n), \; \; X = \sum_{i=1}^n X^i \frac{\partial}{\partial x^i} \label{eq:3.3}
\end{equation}
at some smooth point in local holomorphic coordinates $(x^1, \ldots , x^n )$ of $M$, where $(g_{i \bar{j}})$ is the matrix of $\omega$.

Now, let $M$ be a Fano complete intersection in ${\mathbb CP}^N$ defined by the homogeneous polynomials $F_1, \ldots , F_s$ of degree $d_1, \ldots , d_s$ respectively and suppose that $m \omega \in c_1(M)$ for some constant $m>0$. Let $X$ be a holomorphic vector field tangent to $M$ and $G$ the Lie group generated by $X$. Using the adjunction formula, we know that $m = N+1 -d_1- \cdots - d_s$ and
\begin{equation}
-K_{M_{\rm reg}} \simeq {\mathcal O}(m)|_{M_{\rm reg}} \label{eq:3.4},
\end{equation}
where we remark that this isomorphism is not $G$-equivariant. However, studying the $G$-action on the normal bundle of $M$, Hou \cite[\S.3]{Hou08} (also refer to \cite[Theorem 4.1]{Lu99}) showed that

\begin{lemm}
\label{lemm:3.1}
Let $h$ be the Hermitian metric on ${\mathcal O}(1)$ such that $\omega = - \frac{\sqrt{-1}}{2 \pi} \partial \bar{\partial} \log h$ is a Fubini-Study metric of ${\mathbb CP}^N$ and $V$ a holomorphic vector field such that
\[
V F_i = \alpha_iF_i
\]
for some constants $\alpha_i$ ($i = 1, \ldots , s$). Then we have
\begin{equation}
\mu_{h^m, V} = \sum_{i=1}^s \alpha_i + m \theta_V ,
\end{equation}
where $h^m$ is the Hermitian metric on $-K_{M_{\rm reg}}$ defined via the isomorphism \eqref{eq:3.4}.
\end{lemm}
Let $V$ be a holomorphic vector field defined in Lemma \ref{lemm:3.1}. We set $N_i := \{F_i = 0 \} \subset {\mathbb CP}^N \; (i=1, \ldots s)$, and $M_i := N_1 \cap \cdots \cap N_i \; (i=1, \ldots s)$.
Then we have
\[
M = M_s \subset M_{s-1} \subset \cdots \subset M_1 \subset M_0 := {\mathbb CP}^N.
\]
We define the integrals $I_{k,l}=I_{k,l} (V)$ ($k=0,1, \ldots , s$; $l \geq 0$) by
\begin{equation}
I_{k,l} = m^l \int_{M_k} (\theta_V)^l e^{m \theta_V} \omega^{N-k} \label{eq:3.6},
\end{equation}
\begin{lemm}
\label{lemm:3.2}
For $k=1, \ldots , s$, $I_{k,0}$ satisfies
\begin{equation}
I_{k,0} = \left( d_k -\frac{m \alpha_k}{N-k+1} \right) I_{k-1,0} + \frac{d_k}{N-k+1} I_{k-1,1} \label{eq:3.7}.
\end{equation}
\end{lemm}
\begin{proof}
We can prove \eqref{eq:3.7} in the same way as \cite[Lemma 5.1]{Lu99}.
Define a smooth function $\xi_i$ ($i=1, \ldots , s$) on ${\mathbb CP}^N$ by
\[
\xi_i = \frac{|F_i|^2}{\left( \sum_{i=0}^N |z^i|^2 \right)^{d_i}}.
\]
Using the Poincar\'e-Lelong formula, we obtain
\[
\frac{\sqrt{-1}}{2 \pi} \partial \bar{\partial} \log \xi_k = [N_k] -d_k \omega ,
\]
where $[N_k]$ is the divisor of the zero locus of $F_k$. Then we have
\begin{eqnarray*}
I_{k,0} &=& \int_{M_k} e^{m \theta_V} \omega^{N-k} \\
&=& \int_{M_{k-1}} \left( \frac{\sqrt{-1}}{2 \pi} \partial \bar{\partial} \log \xi_k + d_k \omega \right) \wedge e^{m \theta_V} \omega^{N-k} \\
&=& \int_{M_{k-1}} \frac{\sqrt{-1}}{2 \pi} \partial \bar{\partial} \log \xi_k \wedge e^{m \theta_V} \omega^{N-k} + d_k I_{k-1,0}.
\end{eqnarray*}
On the other hand, using the relation
\[
V \log \xi_k = \alpha_k - d_k \theta_V
\]
and integrating by parts, we obtain
\begin{eqnarray*}
& & \int_{M_{k-1}} \frac{\sqrt{-1}}{2 \pi} \partial \bar{\partial} \log \xi_k \wedge e^{m \theta_V} \omega^{N-k} \\
&=& - \frac{m}{N-k+1} \int_{M_{k-1}} V( \log \xi_k ) e^{m \theta_V} \omega^{N-k+1} \\
&=& - \frac{m \alpha_k}{N-k+1} I_{k-1,0} + \frac{d_k}{N-k+1} I_{k-1,1}.
\end{eqnarray*}
Thus, we get the desired result.
\end{proof}
If we set $V \equiv 0$ and $l=0$, then we obtain
\begin{coro}
\label{coro:3.3}
\begin{equation}
c_1(M)^{N-s} \left( = m^{N-s} \int_M \omega^{N-s} \right) = d_1 \cdots d_s m^{N-s} .
\end{equation}
\end{coro}
In order to get the explicit expression of $I_{k,0}$, we show the next lemma.
\begin{lemm}
\label{lemm:3.4}
For $k=1, \ldots, s$, the equation
\begin{eqnarray}
& &\frac{(N-k)!}{m^{N-k}} \int_{{\mathbb CP}^N} \prod_{i=1}^k (d_i \omega + d_i \theta_V - \alpha_i) e^{m \theta_V} \cdot e^{m \omega} \nonumber \\
&+& \frac{(N-k-1)!}{m^{N-k}}\sum_{i=1}^k  \int_{{\mathbb CP}^N} (d_i \theta_V - \alpha_i) \cdot \prod_{p \in \{1, \ldots , k\} - \{i \}} (d_p \omega + d_p \theta_V - \alpha_p ) e^{m \theta_V} \cdot e^{m \omega} \nonumber \\
&=& \frac{(N-k-1)!}{m^{N-k-1}}  \int_{{\mathbb CP}^N} \prod_{i=1}^k (d_i \omega + d_i \theta_V - \alpha_i) \cdot \omega \cdot e^{m \theta_V} \cdot e^{m \omega} \nonumber \\
\label{eq:3.9}
\end{eqnarray}
holds.
\begin{proof}
For $i=0, \ldots, k$, we define integrals $J_i$ by
\[
J_i:=
\begin{cases}
\int_{{\mathbb CP}^N} \prod_{i=1}^k (d_i \theta_V - \alpha_i) e^{m \theta_V} \omega^N \;\; (\text{when $i=0$}) \\
d_1 \cdots d_k \int_{{\mathbb CP}^N} e^{m \theta_V} \omega^N \;\; (\text{when $i=k$}) \\
\sum_{1 \leq p_1 < \cdots < p_i \leq k} d_{p_1} \cdots d_{p_i}   \int_{{\mathbb CP}^N} (d_{q_1} \theta_V - \alpha_{q_1}) \cdots (d_{q_{k-i}}\theta_V - \alpha_{q_{k-i}}) e^{m \theta_V} \omega^N \;\; (\text{otherwise}),
\end{cases}
\]
where $q_1 < \cdots < q_{k-i}$ and $\{ q_1, \ldots, q_{k-i} \} = \{1, \ldots , k\} - \{ p_1, \ldots, p_i \}$.
Then the direct computation shows that
\[
\frac{(N-k)!}{m^{N-k}} \int_{{\mathbb CP}^N} \prod_{i=1}^k (d_i \omega + d_i \theta_V - \alpha_i) e^{m \theta_V} \cdot e^{m \omega} =
\sum_{i=0}^k \frac{(N-k)! m^{k-i}}{(N-i)!} J_i
\]
and
\begin{eqnarray*}
& &\frac{(N-k-1)!}{m^{N-k}}\sum_{i=1}^k  \int_{{\mathbb CP}^N} (d_i \theta_V - \alpha_i) \cdot \prod_{p \in \{1, \ldots , k\} - \{i \}} (d_p \omega + d_p \theta_V - \alpha_p ) e^{m \theta_V} \cdot e^{m \omega} \\
&=& \sum_{i=0}^k \frac{(N-k-1)! (k-i) m^{k-i}}{(N-i)!} J_i .
\end{eqnarray*}
Hence the left hand side of \eqref{eq:3.9} is
\[
\sum_{i=0}^k \frac{(N-k-1)! m^{k-i}}{(N-i-1)!} J_i,
\]
which is equal to the right hand side of \eqref{eq:3.9}.
\end{proof}
\end{lemm}
\begin{lemm}
\label{lemm:3.5}
For $k=1, \ldots , s$, $I_{k,0}$ can be written as
\begin{equation}
I_{k,0} = \frac{(N-k)!}{m^{N-k}} \int_{{\mathbb CP}^N} \prod_{i=1}^k (d_i \omega + d_i \theta_V - \alpha_i) e^{m \theta_V} \cdot e^{m \omega} \label{eq:3.10}.
\end{equation}
\end{lemm}
\begin{proof}
We will prove \eqref{eq:3.10} by induction for $k$. When $k=1$, the equation \eqref{eq:3.10} coincides exactly with \eqref{eq:3.7}, so the statement holds.

Next, we assume that \eqref{eq:3.10} holds for a fixed $k$. Then, by Lemma \ref{lemm:3.2}, we have
\[
I_{k+1,0} = \left( d_{k+1} - \frac{m \alpha_{k+1}}{N-k} \right) I_{k,0} + \frac{d_{k+1}}{N-k} I_{k,1} .
\]
Since $\theta_{V+tV} = \theta_V + t \theta_V$, $(V+tV)F_i = (\alpha_i+t \alpha_i)F_i$ and $\left. \frac{d}{dt} (d_i \omega + d_i \theta_{V+tV} - \alpha_i- t \alpha_i) \right|_{t=0} = d_i \theta_V - \alpha_i$, using the induction hypothesis, we have
\[
\frac{m \alpha_{k+1}}{N-k} I_{k,0} = \frac{(N-k-1)!}{m^{N-k-1}} \int_{{\mathbb CP}^N} \alpha_{k+1} \prod_{i=1}^k (d_i \omega + d_i \theta_V - \alpha_i) e^{m \theta_V} \cdot e^{m \omega}
\]
and
\begin{eqnarray*}
& &\frac{d_{k+1}}{N-k} I_{k,1} \\
&=& \frac{d_{k+1}}{N-k} \cdot \left. \frac{d}{dt} I_{k,0} (V+tV) \right|_{t=0} \\
&=& d_{k+1} \frac{(N-k-1)!}{m^{N-k}} \sum_{i=1}^k  \int_{{\mathbb CP}^N} (d_i \theta_V - \alpha_i) \cdot \prod_{p \in \{1, \ldots , k\} - \{i \}} (d_p \omega + d_p \theta_V - \alpha_p ) e^{m \theta_V} \cdot e^{m \omega} \\
&+& \frac{(N-k-1)!}{m^{N-k-1}} \int_{{\mathbb CP}^N} d_{k+1} \theta_V \prod_{i=1}^k (d_i \omega + d_i \theta_V - \alpha_i) e^{m \theta_V} \cdot e^{m \omega}.
\end{eqnarray*}
Hence combining with Lemma \ref{lemm:3.4}, we obtain
\begin{eqnarray*}
I_{k+1,0}
&=& d_{k+1} \text{(the LHS of \eqref{eq:3.9})} \\
&+&  \frac{(N-k-1)!}{m^{N-k-1}} \int_{{\mathbb CP}^N} (d_{k+1} \theta_V - \alpha_{k+1})  \prod_{i=1}^k (d_i \omega + d_i \theta_V - \alpha_i) e^{m \theta_V} \cdot e^{m \omega} \\
&=& \frac{(N-k-1)!}{m^{N-k-1}} \int_{{\mathbb CP}^N} \prod_{i=1}^{k+1} (d_i \omega + d_i \theta_V - \alpha_i) e^{m \theta_V} \cdot e^{m \omega}.
\end{eqnarray*}
Hence the statement holds for $k+1$.
\end{proof}
\begin{proof}[Proof of Theorem \ref{theo:1.1}]
By Lemma \ref{lemm:3.1}, ${\mathcal F}$ can by written as
\begin{eqnarray*}
{\mathcal F}(V) &=& - \frac{1}{c_1 (M)^{N-s}} \int_M \exp \left( \sum_{i=1}^s \alpha_i + m \theta_V \right) (m \omega)^{N-s} \\
&=&  - \frac{m^{N-s}}{c_1 (M)^{N-s}} \cdot \exp \left(  \sum_{i=1}^s \alpha_i \right) I_{s,0}.
\end{eqnarray*}
Thus, combining with Corollary \ref{coro:3.3} and Lemma \ref{lemm:3.5}, we get the desired formula for ${\mathcal F}$.
\end{proof}
\section{Another proof of Theorem \ref{theo:1.1}}
In this section, we give another proof of Theorem \ref{theo:1.1} using the algebraic formula for ${\mathcal F}$ (cf: Proposition \ref{prop:2.8}). 
\begin{lemm}[Lemma 5.1 in \cite{AV11}]
\label{lemm:4.1}
Let $B$ be a holomorphic vector bundle of rank $b$ on a manifold $M$, then
\[
\sum_{i=0}^b (-1)^i {\rm ch} (\wedge^i B) = c_b (B) {\rm td}(B)^{-1}.
\]
\end{lemm}
\begin{proof}
Let $r_1, \ldots, r_b$ be the Chern roots of $B$. Since ${\rm ch}(\wedge^i B^{*})= \sum_{1 \leq p_1 < \cdots < p_i \leq b}e^{-(r_{p_1}+ \cdots + r_{p_i})}$,
we obtain
\begin{eqnarray*}
\sum_{i=0}^b (-1)^i {\rm ch}(\wedge^i B^{*}) &=& \sum_{i=0}^b (-1)^i \sum_{1 \leq p_1 < \cdots < p_i \leq b} e^{-(r_{p_1}+ \cdots + r_{p_i})} \\
&=& \prod_{p=1}^b (1- e^{-r_{p}}) \\
&=& \prod_{p=1}^b r_p \prod_{p=1}^b \frac{1- e^{-r_{p}}}{r_p} \\
&=& c_b (B) {\rm td}(B)^{-1}.
\end{eqnarray*}
\end{proof}
Now, let $M$ be an $(N-s)$-dimensioal Fano complete intersection in ${\mathbb CP}^N$, i.e., $M$ is a Fano variety in ${\mathbb CP}^N$ defined by homogeneous polynomials $F_1, \ldots, F_s$, and $V$ a holomorphic vector field on ${\mathbb CP}^N$ tangent to $M$. We will adopt the notation in \S.3. We further assume that $V \in {\mathfrak sl}(N+1,{\mathbb C})$ is a Hermitian matrix so that ${\rm Im}(V)$ is Killing with respect to the Fubini-Study metric $\omega$.
\begin{lemm}[Lemma 5.2 in \cite{AV11}]
\label{lemm:4.2}
We have the following asymptotic expansion of $N_k$ as $k \rightarrow \infty$:
\begin{equation}
N_k = \frac{d_1 \cdots d_s m^{N-s}}{(N-s)!} \cdot k^{N-s} + O(k^{N-s-1}).
\end{equation}
\end{lemm}
\begin{lemm}
\label{lemm:4.3}
We have the following asymptotic expansion of ${\mathcal F}_k (V)$ as $k \rightarrow \infty$:
\begin{equation}
{\mathcal F}_k (V) = - \exp \left( \sum_{i=1}^s \alpha_i \right)  \int_{{\mathbb CP}^N} \prod_{i=1}^{s} (d_i \omega + d_i \theta_V - \alpha_i) e^{m \theta_V} \cdot e^{m \omega} \cdot k^{N-s+1} + O(k^{N-s}).
\end{equation}
\end{lemm}
\begin{proof}
This proof is essentially based on the argument in \cite[Lemma 5.3]{AV11}. The only difference between Lemma \ref{lemm:4.3} and \cite[Lemma 5.3]{AV11} is the linearization of $-K_M$, to which we have only to pay attention. In order to avoid confusion, let $L(\simeq O(m))$ be a linearized line bundle on ${\mathbb CP}^N$ such that $L|_M$ is isomorphic to $-K_M$ as a linearized line bundle whose linearization is determined by the canonical lift of $V/k$ to $-K_M$.

Let ${\mathbb C}_{-\alpha_i/k}$ be a trivial bundle on ${\mathbb CP}^N$ with linearization $t \cdot u = t^{- \alpha_i /k} \cdot u$. Set $L_i := {\mathcal O} (d_i) \otimes {\mathbb C}_{-\alpha_i/k}$ and $B:= L_1 \oplus \cdots \oplus L_s$. Then ${\rm rank}B= s$ and the section $F:=(F_1, \ldots , F_s) \in H^0({\mathbb CP}^N, B)$ is invariant. Since $M$ is complete, the Koszul complex:
\[
0 \rightarrow \wedge^s B^* \rightarrow \wedge^{s-1} B^* \rightarrow \cdots \rightarrow B^* \rightarrow {\mathcal O}_{{\mathbb CP}^N} \rightarrow {\mathcal O}_M \rightarrow 0
\]
is exact and equivariant, where ${\mathcal O}_M$ denotes the structure sheaf of $M$. Tensoring by $L^k$ preserves the exactness and equivariance, so we obtain
\[
\chi^{\mathfrak g}(M, L^k |_M) = \sum_{i=0}^s (-1)^i \chi^{\mathfrak g} ({\mathbb CP}^N, L^k \otimes \wedge^i B^*),
\]
where $\chi^{\mathfrak g}$ denotes the Lefschetz number.
By the equivariant Riemann-Roch formula and Lemma \ref{lemm:4.1}, we get
\begin{eqnarray*}
{\mathcal F}_k (V) &=& -k \sum_{i=0}^s (-1)^i \chi^{\mathfrak g} ({\mathbb CP}^N, L^k \otimes \wedge^i B^*) \\
&=& -k \sum_{i=0}^s (-1)^i \int_{{\mathbb CP}^N} {\rm ch}^{\mathfrak g} (\wedge^i B^*) e^{k c_1 ^{\mathfrak g} (L)} {\rm td}^{\mathfrak g} ({\mathbb CP}^N ) \\
&=& -k \int_{{\mathbb CP}^N} \left( \sum_{i=0}^s (-1)^i {\rm ch}^{\mathfrak g} (\wedge^i B^*) \right) e^{k c_1 ^{\mathfrak g} (L)} {\rm td}^{\mathfrak g} ({\mathbb CP}^N ) \\
&=& -k \int_{{\mathbb CP}^N} c_s ^{\mathfrak g} (B) {\rm td}^{\mathfrak g} (B)^{-1} e^{k c_1 ^{\mathfrak g} (L)} {\rm td}^{\mathfrak g} ({\mathbb CP}^N ) \\
&=& -k \int_{{\mathbb CP}^N} \prod_{i=1}^s \left( d_i c_1^{\mathfrak g}({\mathcal O}(1)) - \frac{\alpha_i}{k} \right) \cdot {\rm td}^{\mathfrak g} (B)^{-1} e^{k c_1 ^{\mathfrak g} (L)} {\rm td}^{\mathfrak g} ({\mathbb CP}^N ).
\end{eqnarray*}
Let $h$ be a Hermitian metric on ${\mathcal O}(1)$ such that $\omega = - \frac{\sqrt{-1}}{2 \pi} \partial \bar{\partial} \log h$ is the Fubini-Study metric of the ${\mathbb CP}^N$. Then, by Lemma \ref{lemm:3.1}, the equivariant 1st Chern form for $(h, V/k)$ and $(h^m, V/k)$ are written as
\[
\omega + \frac{1}{k} \theta_V \in c_1^{\mathfrak g} ({\mathcal O} (1)) \; \; \text{and} \; \; m \omega +  \frac{m}{k} \theta_V + \frac{1}{k} \sum_{i=1}^s \alpha_i \in c_1^{\mathfrak g} (L)
\]
respectively. Both ${\rm td}^{\mathfrak g} (B)^{-1}$ and  ${\rm td}^{\mathfrak g} ({\mathbb CP}^N )$ can be written as the form
\[
1+A+\sum_{i \geq 1} \frac{1}{k^i} B_i,
\]
where $A$ (resp. $B_i$) denotes $2l$-forms ($l \geq 1$ (resp. $l \geq 0$)) not depending on $k$. Hence we have
\begin{eqnarray*}
{\mathcal F}_k (V) &=& -k \exp \left( \sum_{i=1}^s \alpha_i \right) \int_{{\mathbb CP}^N} \prod_{i=1}^s \left( d_i \omega + \frac{1}{k} (d_i \theta_V - \alpha_i) \right)  {\rm td}^{\mathfrak g} (B)^{-1} e^{m \theta_V} \cdot e^{km \omega} {\rm td}^{\mathfrak g} ({\mathbb CP}^N ) \\
&=&  - \exp \left( \sum_{i=1}^s \alpha_i \right)  \int_{{\mathbb CP}^N} \prod_{i=1}^{s} (d_i \omega + d_i \theta_V - \alpha_i) e^{m \theta_V} \cdot e^{m \omega} \cdot k^{N-s+1} + O(k^{N-s}).
\end{eqnarray*}
\end{proof}
\begin{proof}[Proof of Theorem \ref{theo:1.1}]
By Lemma \ref{lemm:4.2} and Lemma \ref{lemm:4.3}, we have an asymptotic expansion as $k \rightarrow \infty$:
\[
 \frac{1}{kN_k} {\mathcal F}_k (V) = - \frac{(N-s)!}{d_1 \cdots d_s m^{N-s}} \exp \left( \sum_{i=1}^s \alpha_i \right) \int_{{\mathbb CP}^N} \prod_{i=1}^s (d_i \omega + d_i \theta_V - \alpha_i) e^{m \theta_V} \cdot e^{m \omega} + O(k^{-1}).
\]
On the other hand,  by Proposition \ref{prop:2.8} (4), $\frac{1}{kN_k} {\mathcal F}_k (V)$ converges to ${\mathcal F} (V)$ as $k \rightarrow \infty$. Hence we have the desired formula.
\end{proof}
\section{Examples}
In this section, we compute ${\mathcal F}$ for several examples in \cite[\S.6]{Lu99}.
Let $M$ be a Fano complete intersection in ${\mathbb CP}^N$. We will adopt the notation in \S.3. First, we will mention some results obtained as a corollary of the localization formula in holomorphic equivariant cohomology theory (cf: \cite[Theorem 1.6]{Liu95}).
\begin{lemm}
\label{lemm:5.1}
If $V={\rm diag}(\lambda_0 ,\ldots, \lambda_N)$ is a diagonal matrix with different eigenvalues $\lambda_0, \ldots, \lambda_N$. Then we have
\begin{equation}
I_{0,0} = N! \sum_{i=0}^N \frac{e^{m \lambda_i}}{\prod_{p \in \{0, \ldots, N \} -\{ i \} } (\lambda_i -\lambda_p)} \label{eq:5.1}.
\end{equation}
\end{lemm}
Since $I_{0,l}$ are given by the derivatives of $I_{0,0}$, we can compute $I_{0,l}$ for any integer $l$.
On the other hand, by Theorem \ref{theo:1.1}, ${\mathcal F}(V)$ can be written as a linear combination of $I_{0,l}$ ($0 \leq l \leq s$). Hence we can express ${\mathcal F}(V)$ in terms of the eigenvalues of $V$.

However, we can calculate ${\mathcal F}(V)$ without using Theorem \ref{theo:1.1} in a special case: we assume that $M$ has at worst orbifold singularities and $V$ satisfies the condition:\\
(1) $V$ has isolated zero points $\{ p_i \}$.\\
(2) $V$ is nondegenerate at each zero point $p_i$, i.e., for each local uniformization $\pi : U \rightarrow U/{\Gamma_i} \subset M$ with $\pi(U) \cap p_i \neq \emptyset$, $\pi^* V$ vanishes along $\pi^{-1}(p_i)$ and the matrix $B_i = \left( - \frac{\partial v_j ^i}{\partial z^k} \right)_{1 \leq j, k \leq N-s}$ is nondegenerate near $\pi^{-1}(p_i)$, where $(z^1, \ldots , z^{N-s})$ is local holomorphic coordinates around $\pi^{-1}(p_i)$ and $V = \sum_{j=1}^{N-s} v_j^i \frac{\partial}{\partial z^j}$.

In the same way as \cite[Proposition 1.2]{DT92}, we have
\begin{lemm}
\label{lemm:5.2}
Let $M$ and $V$ be as above. Then we have
\begin{equation}
{\mathcal F}(V) = - \frac{(N-s)!}{d_1 \cdots d_s} \exp \left( \sum_{i=1}^s \alpha_i \right) \cdot \sum_i \frac{1}{|\Gamma_i |} \cdot \frac{e^{m \theta_V (p_i)}}{\det B_i} \label{eq:5.2},
\end{equation}
where $|\Gamma_i |$ is the order of the local uniformization group $\Gamma_i$ at a point $p_i$.
\end{lemm}
\begin{rema}
One can extend Lemma \ref{lemm:5.1} and Lemma \ref{lemm:5.2} to the case when the zero set of $V$ is the sum of nondegenerate submanifolds, where the word {\it nondegenerate} means that the induced actions of $V$ to the normal bundle of submanifolds are nondegenerate. However, since $I_{0,0}(V)$ and ${\mathcal F}(V)$ are clearly continuous with respect to $V$, we may think that the equations \eqref{eq:5.1} and \eqref{eq:5.2} hold in the sense of limit $V_{\epsilon} \rightarrow V$ of any expression. For instance,
\end{rema}
\begin{lemm}
\label{lemm:5.4}
Let $m=1$ and $V={\rm diag}(\lambda_0,\lambda_1,\lambda_2,\lambda_2) \in {\mathfrak sl}(4,{\mathbb C})$ be a holomorphic vector field on ${\mathbb CP}^3$, where $\lambda_0$, $\lambda_1$ and $\lambda_2$ are different numbers. Then we have
\begin{eqnarray}
I_{0,0} &=& 6 \left[ \frac{e^{\lambda_0}}{(\lambda_0-\lambda_1) (\lambda_0 - \lambda_2)^2} + \frac{e^{\lambda_1}}{(\lambda_1-\lambda_0)(\lambda_1-\lambda_2)^2} \right. \nonumber \\
&+& \left. \frac{\{ \lambda_0 + \lambda_1 -2 \lambda_2 + (\lambda_2 - \lambda_0)(\lambda_2 - \lambda_1) \} e^{\lambda_2}}{(\lambda_2 - \lambda_0)^2 (\lambda_2-\lambda_1)^2} \right].
\end{eqnarray}
\begin{proof}
Let $\epsilon \neq 0$ be a small number. if we set $V_{\epsilon} := {\rm diag}(\lambda_0, \lambda_1, \lambda_2 + \epsilon, \lambda_2 - \epsilon)$, then $V_{\epsilon}$ has different eigenvalues. Hence we can compute $I_{0,0}(V) = \lim_{\epsilon \rightarrow 0} I_{0,0} (V_{\epsilon})$ directly using \eqref{eq:5.1}.
\end{proof}
\end{lemm}
\begin{exam}
\label{exam:5.5}
Let $M \subset {\mathbb CP}^3$ be the zero set of a cubic polynomial $F:= z_0 z_1^2 + z_2 z_3(z_2 - z_3)$, where $(z_0, z_1, z_2, z_3)$ are homogeneous coordinates of ${\mathbb CP}^3$ and $V = {\rm diag}(-7t,5t,t,t)$ $(t \neq 0)$ a holomorphic vector field tangent to $M$. We compute ${\mathcal F}$ in two methods:\\
(1) The variety $M$ has a unique quotient singularity at $p_0 :=[1,0,0,0]$. If we restricts $V$ to $M$, $V$ has five zeros $p_0=[1,0,0,0]$, $[0,1,0,0]$, $[0,0,1,0]$, $[0,0,0,1]$ and $[0,0,1,1]$.
Let $\zeta_i := \frac{z_i}{z_0}$ $(i=1,2,3)$ be Euclidean coordinates defined near $p_0$. Then we can rewrite $F$ near $p_0$ in the standard form
\[
f= \frac{F}{z_0^3} = \zeta_1^2 - \zeta_3 (\zeta_2 ^2 - 4 \zeta_3 ^2).
\]
According to \cite[Example 1]{Lu99}, we see that there is a uniformization $\phi:{\mathbb C}^2 \rightarrow {\mathbb C}^2 / \Gamma \subset M$ defined by
\[
\phi:
\begin{cases}
\zeta_1 = uv (u^4 - v^4) \\
\zeta_2 = u^4 + v^4 \\
\zeta_3 = u^2 v^2 ,
\end{cases}
\]
where $\Gamma$ is the dihedral subgroup in $SU(2)$ of type $D_4$.
Thus, we have $\phi^* (V)=2t u \frac{\partial}{\partial u} + 2t v \frac{\partial}{\partial v}$. Since the order of the group $D_4$ is $8$, applying Lemma \ref{lemm:5.2}, we obtain
\begin{eqnarray*}
{\mathcal F}(V) &=& - \frac{2}{3} e^{3t} \left( \frac{1}{8} \cdot \frac{e^{-7t}}{4t^2} + \frac{e^{5t}}{16t^2} + 3 \cdot \frac{e^{t}}{-32t^2} \right) \\
&=& - \frac{e^{-4t}}{48t^2} - \frac{e^{8t}}{24t^2} + \frac{e^{4t}}{16t^2}.
\end{eqnarray*}
(2) By Theorem \ref{theo:1.1}, we obtain
\begin{eqnarray*}
{\mathcal F}(V) &=& -\frac{2}{3} e^{3t} \int_{{\mathbb CP}^3} (3 \omega + 3 \theta_V - 3t) e^{\theta_V} e^{\omega} \\
&=& - e^{3t} \left\{ \left( 1 - \frac{t}{3} \right) I_{0,0} + \frac{1}{3} I_{0,1} \right\}.
\end{eqnarray*}
By Lemma \ref{lemm:5.4}, we have
\[
I_{0,0} = - \frac{e^{-7t}}{128t^3} + \frac{e^{5t}}{32t^3} - \frac{3(1+8t)e^{t}}{128t^3}
\]
and
\[
I_{0,1} = \frac{(7t+3)e^{-7t}}{128t^3} + \frac{(5t-3)e^{5t}}{32t^3} - \frac{3(8t^2-15t-3)e^{t}}{128t^3}.
\]
Hence we have
\[
{\mathcal F}(V) = - \frac{e^{-4t}}{48t^2} - \frac{e^{8t}}{24t^2} + \frac{e^{4t}}{16t^2}.
\]
\end{exam}
\begin{exam}
\label{exam:5.6}
Let $M \subset {\mathbb CP}^4$ be the zero locus defined by
\[
\begin{cases}
F_1 = z_0 z_1 + z_2 ^2 \\
F_2 = z_1^2 + z_3 z_4
\end{cases}
\]
and $V={\rm diag}(-7t, 3t, -2t, 5t, t)$ $(t \neq 0)$ a holomorphic vector field tangent to $M$. 
In the same way as (2) in Example \ref{exam:5.5}, we get
\[
{\mathcal F}(V)= - e^{2t} \left\{ \left(1- \frac{t}{3}-\frac{t^2}{2} \right) I_{0,0} + \left( \frac{2}{3} - \frac{t}{12} \right) I_{0,1} + \frac{1}{12} I_{0,2} \right\} ,
\]
\[
I_{0,0} = \frac{e^{-7t}}{200t^4} - \frac{3 e^{3t}}{25t^4} - \frac{24 e^{-2t}}{525t^4} + \frac{e^{5t}}{28t^4} + \frac{e^{t}}{8t^4} ,
\]
\[
I_{0,1} = - \frac{(7t+ 4)e^{-7t}}{200t^4} + \frac{3(4-3t) e^{3t}}{25t^4} + \frac{48(t+2) e^{-2t}}{525t^4} + \frac{(5t-4)e^{5t}}{28t^4} + \frac{(t-4)e^{t}}{8t^4}
\]
and
\begin{eqnarray*}
I_{0,2} &=& \frac{(49t^2+56t+20)e^{-7t}}{200t^4} - \frac{3(9t^2-24t+20) e^{3t}}{25t^4} - \frac{96(t^2+4t+5) e^{-2t}}{525t^4} \\
&+& \frac{5(5t^2-8t+4)e^{5t}}{28t^4} + \frac{(t^2-8t+20)e^{t}}{8t^4}.
\end{eqnarray*}
Hence we have
\[
{\mathcal F}(V) = - \frac{e^{-5t}}{48t^2} - \frac{e^{7t}}{24t^2} + \frac{e^{3t}}{16t^2} .
\]
Here we remark that $V$ has only three zero points $p_1=[1,0,0,0,0]$, $p_2=[0,0,0,1,0]$, $p_3=[0,0,0,0,1]$ in $M$. Actually, the exponents appeared in the above expression of ${\mathcal F}(V)$ are $-5t=\theta_V(p_1)+2t$, $7t= \theta_V (p_2) + 2t$, $3t= \theta_V (p_3) + 2t$, hence correspond to the three zero points of $V$.
\end{exam}

\end{document}